\documentclass{svproc} 

\usepackage{graphicx}
\usepackage{amsmath}
\usepackage{amssymb}
\usepackage{amsfonts}
\usepackage{enumitem}
\usepackage{xcolor}
\usepackage{cite}
\usepackage{multirow}
\usepackage{mathtools}

\usepackage{url}

\begin{document}
\mainmatter
\title{A reduced scalar potential approach for magnetostatics avoiding the coenergy}

\titlerunning{A modified reduced scalar potential approach}
\author{Herbert Egger\inst{1,2}, Felix Engertsberger\inst{1} \and
Bogdan Radu\inst{2}}
\authorrunning{Herbert Egger et al.}
\institute{Institute for Numerical Mathematics, Johannes-Kepler University Linz, Austria\\
\email{herbert.egger@jku.at}, \email{felix.engertsberger@jku.at},
\and Johann Radon Institute for Computational and Applied Mathematics, Linz, Austria\\
\email{bogdan.radu@ricam.oeaw.ac.at}
}

\maketitle

\begin{abstract}
The numerical solution of problems in nonlinear magnetostatics is typically based on a variational formulation in terms of magnetic potentials, the discretization by finite elements, and iterative solvers like the Newton method.  
The vector potential approach aims at minimizing a certain energy functional and, in three dimensions, requires the use of edge elements and appropriate gauging conditions.
The scalar potential approach, on the other hand, seeks to maximize the negative coenergy and can be realized by standard Lagrange finite elements, thus reducing the number of degrees of freedom and simplifying the implementation. 
The number of Newton iterations required to solve the governing nonlinear system, however, has been observed to be usually higher than for the vector potential formulation. 
In this paper, we propose a method that combines the advantages of both approaches, i.e., it requires as few Newton iterations as the vector potential formulation while involving the magnetic scalar potential as the primary unknown. 
We discuss the variational background of the method, its well-posedness, and its efficient implementation. 
Numerical examples are presented for illustration of the accuracy and the gain in efficiency compared to other approaches.
\end{abstract}

\section{Introduction and motivation}
The equilibrium distributions of the magnetic fields generated by an imposed current density $\mathbf{j}_s = \operatorname{curl} \mathbf{h}_s$ in a nonlinear magnetically isolated medium satisfy
\begin{alignat}{2}
\partial_\mathbf{b} w(\mathbf{b}) + \nabla \psi &= \mathbf{h}_s \qquad && \text{in } \Omega, \label{radu:eq:1} \\
\operatorname{div} \mathbf{b} &= 0 \qquad && \text{in } \Omega \label{radu:eq:2}\\
\mathbf{b} \cdot \mathbf{n} &= 0 \qquad && \text{on } \partial\Omega. \label{radu:eq:3}
\end{alignat}
Here $\mathbf{b}$ is the magnetic flux density, $\mathbf{h}_s$ the source field, and $\psi$ the magnetic scalar potential. The magnetic field intensity $\mathbf{h}=\partial_b w(\mathbf{b})$ is obtained as the derivative of the magnetic energy density $w(\mathbf{b})=w(\mathbf{b},x)$, which describes the local nonlinear material response~\cite{radu:Silvester1991} and may in general depend on the spatial position. 
The first equation thus amounts to the decomposition 
$\mathbf{h} = \mathbf{h}_s - \nabla \psi$
and Ampere's law $\operatorname{curl} \mathbf{h} = \mathbf{j}_s$ can be recovered by application of the $\operatorname{curl}$ operator.

\subsection*{Vector potential formulation}
The magnetic Gauß law~\eqref{radu:eq:2} can be satisfied by representing $\mathbf{b} = \operatorname{curl} \mathbf{a}$ in terms of a magnetic vector potential. The remaining equations can then be restated as
\begin{alignat}{2}
\operatorname{curl} (\partial_\mathbf{b} w(\operatorname{curl} \mathbf{a})) &= \mathbf{j}_s \qquad && \text{in } \Omega \label{radu:eq:4} \\
\mathbf{a} \times \mathbf{n} &= 0 \qquad && \text{on } \partial\Omega. \label{radu:eq:5}
\end{alignat}
We tacitly assumed that the boundary $\partial\Omega$ is simply connected and further applied the $\operatorname{curl}$ operator to \eqref{radu:eq:1} as before. Hence \eqref{radu:eq:4} amounts to Ampere's law stated in terms of the magnetic vector potential. 
The weak formulation of this problem is well-suited for analysis and discretization by $H(\operatorname{curl})$-conforming finite elements~\cite{radu:Meunier,radu:Nedelec1980}, and the resulting nonlinear systems can be solved efficiently by Newton-type methods~\cite{radu:Egger2024}. 
Appropriate gauging is required to ensure the uniqueness of the vector potential on the continuous and the discrete level~\cite{radu:Albanese1997,radu:Hiptmair2002}. 

\subsection*{Scalar potential formulation}
Let $w_*(\mathbf{h}) = \sup_\mathbf{b}\{\langle \mathbf{h},\mathbf{b}\rangle  - w(\mathbf{b})\}$ denote the convex conjugate coenergy density. This allows to recast the material law $\mathbf{h} = \partial_\mathbf{b} w(\mathbf{b})$ in inverse form $\mathbf{b} = \partial_\mathbf{h} w_*(\mathbf{h})$. Using $\mathbf{h}  = \mathbf{h}_s - \nabla \psi$, which also encodes Ampere's law, and invoking the magnetic Gau{\ss} law~\eqref{radu:eq:2} then leads to the nonlinear scalar potential formulation
\begin{alignat}{2}
\operatorname{div} (\partial_\mathbf{h} w_*(\mathbf{h}_s - \nabla \psi)) &= 0 \qquad && \text{in } \Omega \label{radu:eq:6} \\
\mathbf{b} \cdot \mathbf{n} &= 0 \qquad && \text{on } \partial\Omega. \label{radu:eq:7}
\end{alignat}
The weak form of this problem is again well-suited for the analysis of the problem. $H^1$-conforming finite elements can now be used for the systematic discretization and Newton-type methods can be employed for the iterative solution of the resulting nonlinear systems. We refer to \cite{radu:Engertsberger2023,radu:Meunier} for details and further information. 

\begin{remark}
In three space dimensions, the scalar potential approach achieves the same approximation order with much fewer degrees of freedom and, additionally, avoids the need for a complex gauging procedure. 
The two formulations mentioned above, however, involve different nonlinear functions and the vector potential approach seems to require fewer Newton iterations in typical cases~\cite{radu:Dular2020}. 
\end{remark}

\subsection*{Outline of the main results}
We start from the mixed formulation \eqref{radu:eq:1}--\eqref{radu:eq:3}, but instead of treating \eqref{radu:eq:1} in strong form, like in previous works~\cite{radu:Dular1997,radu:Meunier}, we here use the weak form of this constraint and keep the scalar potential term $\nabla \psi$ in \eqref{radu:eq:1} in strong form. 
The magnetic Gauß law~\eqref{radu:eq:2}, on the other hand, is treated weakly. 
Then $\mathbf{b}$ can be approximated by discontinuous finite elements which allows to eliminate this variable locally. 
A key point of our approach is to do this only after linearization, i.e., in the solution of the linear saddle point problems arising from the Newton iteration.
Overall, we obtain a method that uses the same linearization $\partial_\mathbf{b} w(\mathbf{b})$ of the energy density as the vector potential formulation, while the linear systems to be solved in every step of the Newton iteration are of the same structure and can be solved at the same cost as those of the scalar potential formulation.
In summary, this leads to a modified scalar potential formulation with reduced computational cost.

\section{Preliminaries}

Let $\Omega \subset \mathbb{R}^3$ be a bounded Lipschitz domain and simply connected. 
Standard notation for Lebesgue and Sobolev spaces is used throughout the text; see \cite{radu:Monk} for details. 
We write $\langle a,b\rangle$ for the scalar product in $\mathbb{R}^3$ and $\langle a,b\rangle_\Omega = \int_\Omega \langle a,b\rangle \, dx$ for the $L^2$-scalar product on $\Omega$. 
We further assume that 
$w:\mathbb{R}^3\times\Omega \to \mathbb{R}$ is piecewise smooth in the first argument and smooth in the second with strongly monotone and Lipschitz continuous derivative, i.e., 
\begin{alignat}{2}
\langle\partial_{\mathbf{b}} w(x,\mathbf{b})-\partial_{\mathbf{b}} w(x,\mathbf{b}'),\mathbf{b}-\mathbf{b}'\rangle & \ge \gamma |\mathbf{b}-\mathbf{b}'|^2, \label{radu:eq:wa1} \\
|\partial_{\mathbf{b}} w(x,\mathbf{b})-\partial_{\mathbf{b}} w(x,\mathbf{b}')| &\le L |\mathbf{b}-\mathbf{b}'| \label{radu:eq:wa2}
\end{alignat}
for all $x \in \Omega$ and $\mathbf{b},\mathbf{b}' \in \mathbb{R}^3$ with uniform constants $L, \gamma>0$.
We will not mention the spatial dependence explicitly in the remainder of the manuscript.
As a final assumption, we require $\mathbf{h}_s \in H(\operatorname{curl};\Omega)$ and set $\mathbf{j}_s = \operatorname{curl} \mathbf{h}_s$.

\subsection*{Well-posedness and equivalence}
For completeness of the presentation, let us summarize some basic facts about the well-posedness and mutual connections of the different formulations. 

\begin{lemma}
Under the above assumptions, the problem \eqref{radu:eq:1}--\eqref{radu:eq:3} has a unique solution $(\mathbf{b},\psi)$ and $\mathbf{h}=\mathbf{h}_s - \nabla \psi$ satisfies $\operatorname{curl} \mathbf{h} = \mathbf{j}_s$. 
Moreover, $\psi$ is also the unique solution of \eqref{radu:eq:6}--\eqref{radu:eq:7} and any $\mathbf{a}$ solving \eqref{radu:eq:4}--\eqref{radu:eq:5} also satisfies $\mathbf{b} = \operatorname{curl} \mathbf{a}$.
\end{lemma}
The main arguments to prove existence of a unique solution for the scalar potential equation can be found in \cite{radu:Engertsberger2023}. The remaining details are left to the reader.

\subsection*{Discrete mixed problem}

Let $\mathcal{T}_h$ be a tetrahedral mesh of $\Omega$ satisfying the usual regularity properties~\cite{radu:Ciarlet,radu:Monk}. 

We will make use of the finite element spaces
\begin{alignat}{2}
Q_h &= P_{p-1}(\mathcal{T}_h)^d 
\quad \text{and} \quad 
& V_h &= P_p(\mathcal{T}_h) \cap H^1(\Omega)/\mathbb{R},
\end{alignat}
consisting of all vector-valued piecewise polynomials of degree $\le p-1$ and all continuous piecewise polynomials of degree $\le p$ having zero average, respectively.
We further denote by $\langle a,b\rangle_h = \sum_T \langle a,b\rangle_{T,h}$ an approximation of the standard $L^2$-scalar product $\langle a,b\rangle_\Omega$ obtained by a local quadrature rule applied on every element. This quadrature rule is assumed to have positive weights and to integrate polynomials of degree $\le 2p-2$ exactly. 

For the discretization of \eqref{radu:eq:1}--\eqref{radu:eq:3}, we then consider the following method.
\begin{problem} \label{radu:prob:1}
Find $\mathbf{b}_h \in Q_h$ and $\psi_h \in V_h$ such that 
\begin{alignat}{2}
\langle \partial_\mathbf{b} w(\mathbf{b}_h), \mathbf{b}_h' \rangle_h + \langle \nabla \psi_h', \mathbf{b}_h'\rangle_h &= \langle \mathbf{h}_s, \mathbf{b}_h' \rangle_h \qquad &&  \forall \mathbf{b}_h \in Q_h \label{radu:eq:11}\\
\langle \mathbf{b}_h, \nabla \psi_h\rangle_h &= 0 \qquad &&  \forall \psi_h' \in V_h. \label{radu:eq:12}
\end{alignat}
\end{problem}
With similar arguments as for the continuous problem, we obtain the following.
\begin{lemma}
Under the above assumptions, Problem~\ref{radu:prob:1} admits a unique solution. 
\end{lemma}
\begin{proof}
Let us note that \eqref{radu:eq:11}--\eqref{radu:eq:12} corresponds to the first-order optimality condition for the constrained minimization problem
\begin{align}
&\min_{\mathbf{b}_h \in Q_h}  J_h(\mathbf{b}_h):=\langle w(\mathbf{b}_h),1\rangle_h - \langle \mathbf{h}_s, \mathbf{b}_h\rangle_h  \label{radu:eq:13}\\
& \qquad \qquad \quad  \,\text{s.t.} \quad \langle \mathbf{b}_h,\nabla \psi_h'\rangle_h=0 \quad  \forall \psi_h' \in V_h. \notag
\end{align}
Under our assumptions, the functional $J_h$ to be minimized is strictly convex and coercive. 
%
This guarantees the existence of a unique solution $\mathbf{b}_h \in Q_h$. The existence of a unique discrete Lagrange multiplier $\psi_h$ follows from the compatibility of the discretization spaces; see~\cite{radu:Boffi} for related results in the linear case. \qed
\end{proof}

\subsection*{Scalar and vector potential formulations}
For later reference let us briefly mention the basic form of the discretization schemes for the other approaches mentioned in the introduction. 
The approximation for the scalar potential problem reads: Find $\psi_h \in V_h$ such that 
\begin{align} \label{radu:eq:14}
\langle \partial_\mathbf{h} w_*( \mathbf{h}_s - \nabla \psi_h), \nabla \psi_h'\rangle_h &= 0 \qquad  \forall \psi_h' \in V_h.
\end{align}
The existence of a unique solution again follows by standard arguments~\cite{radu:Engertsberger2023}. 
In our numerical tests, we also compare to a discretization of the vector potential formulation, which reads: Find $\mathbf{a}_h \in W_h$ such that 
\begin{align} \label{radu:eq:15}
\langle \partial_\mathbf{b} w(\operatorname{curl} \mathbf{a}_h), \operatorname{curl} \mathbf{a}_h'\rangle_h &= \langle \mathbf{h}_s, \operatorname{curl} \mathbf{a}_h'\rangle_h \qquad  \forall a_h' \in W_h.
\end{align}
Here $W_h \subset H_0(\operatorname{curl};\Omega)$ is an appropriate sub-space of edge elements incorporating the required gauging conditions.
The existence of a unique discrete solution can then again be established by well-known arguments~\cite{radu:Heise1994}.

\section{Iterative solution algorithm}

For the iterative solution of \eqref{radu:eq:11}--\eqref{radu:eq:12}, we consider a Newton method with line search. Every step of this iteration has the form 
\begin{align} \label{radu:eq:16}
(\mathbf{b}_h^{n+1},\psi_h^{n+1}) &= (\mathbf{b}_h^{n},\psi_h^n) + \tau^n (\delta \mathbf{b}_h^n,\delta \psi_h^n)
\end{align}
with increments $(\delta \mathbf{b}_h,\delta \psi_h) \in Q_h \times V_h$  determined by the linearized systems
\begin{alignat}{2}
\langle \partial_{bb} w(\mathbf{b}_h^n) \delta \mathbf{b}_h^n, \mathbf{b}_h'\rangle_h + \langle \nabla \delta \psi_h^n,\mathbf{b}_h'\rangle_h &= -\langle \partial_b w(\mathbf{b}_h^n) + \nabla \psi_h^n -\mathbf{h}_s, \mathbf{b}_h'\rangle_h \label{radu:eq:17}\\
\langle \delta \mathbf{b}_h^n,\nabla \psi_h'\rangle_h &= -\langle \mathbf{b}_h^n,\nabla \psi_h\rangle_h . \label{radu:eq:18}
\end{alignat}
for all  test functions $\mathbf{b}_h' \in Q_h$ and $\psi_h' \in V_h$. 
With standard results about linear saddlepoint  problems~\cite{radu:Boffi}, one can show the following result. 
\begin{lemma}
For any given $(\mathbf{b}_h^n,\psi_h^n) \in Q_h \times V_h$, the system \eqref{radu:eq:17}--\eqref{radu:eq:18} has a unique solution $(\delta \mathbf{b}_h^n, \delta \psi_h^n) \in Q_h \times V_h$. 
Moreover, if $\langle \mathbf{b}_h^n,\nabla \psi_h'\rangle_h=0$ for all $\psi_h' \in V_h$, then one also has $\langle \delta \mathbf{b}_h^n,\nabla \psi_h'\rangle_h = 0$ for all $\psi_h' \in V_h$.
\end{lemma}

\begin{remark}
Together with an appropriate line search strategy, one can now establish global convergence of the iteration \eqref{radu:eq:16}--\eqref{radu:eq:17} using well-known arguments from nonlinear optimization.
In our numerical tests, we use Armijo backtracking~\cite{radu:Kelley}, which allows to guarantee sufficient decay in the energy \eqref{radu:eq:13} in every step. 
If assuming the initial choice $\mathbf{b}_h^0=0$, global linear convergence with a mesh-independent convergence rate follows from the arguments in \cite{radu:Egger2024}. 
Similar results can also be proven for the other discretization strategies discussed above.
\end{remark}

\section{Efficient realization of the Newton step}
In every step of the proposed iterative method, a linearized discrete saddle point problem \eqref{radu:eq:17}--\eqref{radu:eq:18} has to be solved. 
On the algebraic level, it has the form
\begin{align}
\begin{pmatrix*}
A & \, \, B^\top \\
B & 0 
\end{pmatrix*}
\begin{pmatrix}
db \\ d\psi 
\end{pmatrix}
= 
\begin{pmatrix}
f \\ g
\end{pmatrix}
\end{align}
To avoid misunderstanding, we use slightly different notation for the coordinate vectors $(db,d\psi)$ than for the corresponding functions $(\delta \mathbf{b}_h^n,\delta \psi_h^n)$. 
Due to our assumptions, the matrix $A$ can be shown to be positive definite, and since the finite element space $Q_h$ carries no continuity properties, it has a block-diagonal structure. 
One can thus efficiently assemble the Schur complement system 
\begin{align}
B A^{-1} B^\top d\psi &= B A^{-1} f - g
\end{align}
whose system matrix $S=B A^{-1} B^\top$ corresponds to the stiffness matrix of a standard finite element discretization for the Poisson problem. In particular, it has the same properties as that of the discretized and linearized scalar potential problem~\eqref{radu:eq:14}. 
In a second step, one can determine the second component 
\begin{align}
db = A^{-1} (f - B^\top d\psi),  
\end{align}
which can again be computed efficiently, since $A^{-1}$ is still block diagonal. In summary, 
the Newton-update $(db,d\psi)$ can be computed with a computational effort comparable to that of the standard scalar potential formulation.

\section{Numerical results}

To illustrate our theoretical results, we now present some computational tests in which we compare the proposed methods with the standard scalar- and vector potential formulations. 
As a test problem, we use Problem~13 of the Compumag TEAM suite~\cite{radu:TEAM13}. The geometry for this problem consists of a coil surrounded by steel plates; see Figure~\ref{radu:fig:t13} for a sketch.
\begin{figure}
    \centering
    \includegraphics[scale=0.35]{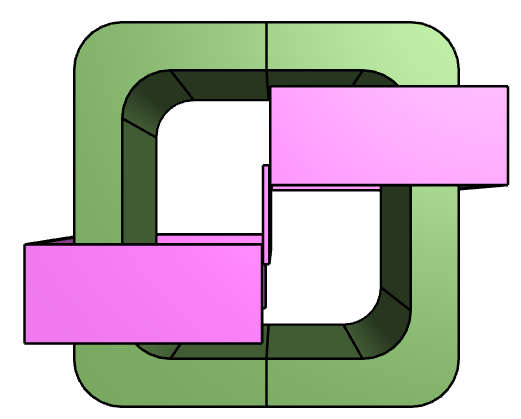}
    \caption{Bird's eye view of the geometry of TEAM problem 13. Pink volumes are the steel plates; the green area is the coil. The ambient air box is hidden.}
    \label{radu:fig:t13}
\end{figure}

\subsubsection*{Material laws.} 
In the air and the coil, the magnetic response is modeled by the linear material relation $\mathbf{b}=\mu_0 \mathbf{h}$, where $\mu_0$ is the vacuum permeability.
This corresponds to $w(\mathbf{b})=\frac{2}{2\mu_0}|\mathbf{b}|^2$ respectively $w_*(\mathbf{h})=\frac{\mu_0}{2} |\mathbf{h}|^2$. 
For the steel plates we us an isotropic nonlinear relation $ w(\mathbf{b})=\widetilde w(|\mathbf{b}|)$. Then $\widetilde w'(|\mathbf{b}|) = |\mathbf{h}|$ can be obtained from the B-H data provided in \cite[Figure~3]{radu:TEAM13}. 
The functions $\widetilde w(|\mathbf{b}|)$ and $\widetilde w''(|\mathbf{b}|)$ which are required to compute $w(\mathbf{b})$ and $\partial_{\mathbf{bb}} w(\mathbf{b})$ for the nonlinear iteration can then be obtained by integration resp. differentiation. 
Validity of the conditions \eqref{radu:eq:wa1}--\eqref{radu:eq:wa2} can be guaranteed explicitly by this procedure. 
The setup of the coenergy density $w_*(\mathbf{h})=\widetilde w_*(|\mathbf{h}|)$ is done in a similar manner. 
The solution of our test problem is depicted in Figure~\ref{radu:fig:2}. 
\begin{figure}
    \centering
    \includegraphics[scale=0.38]{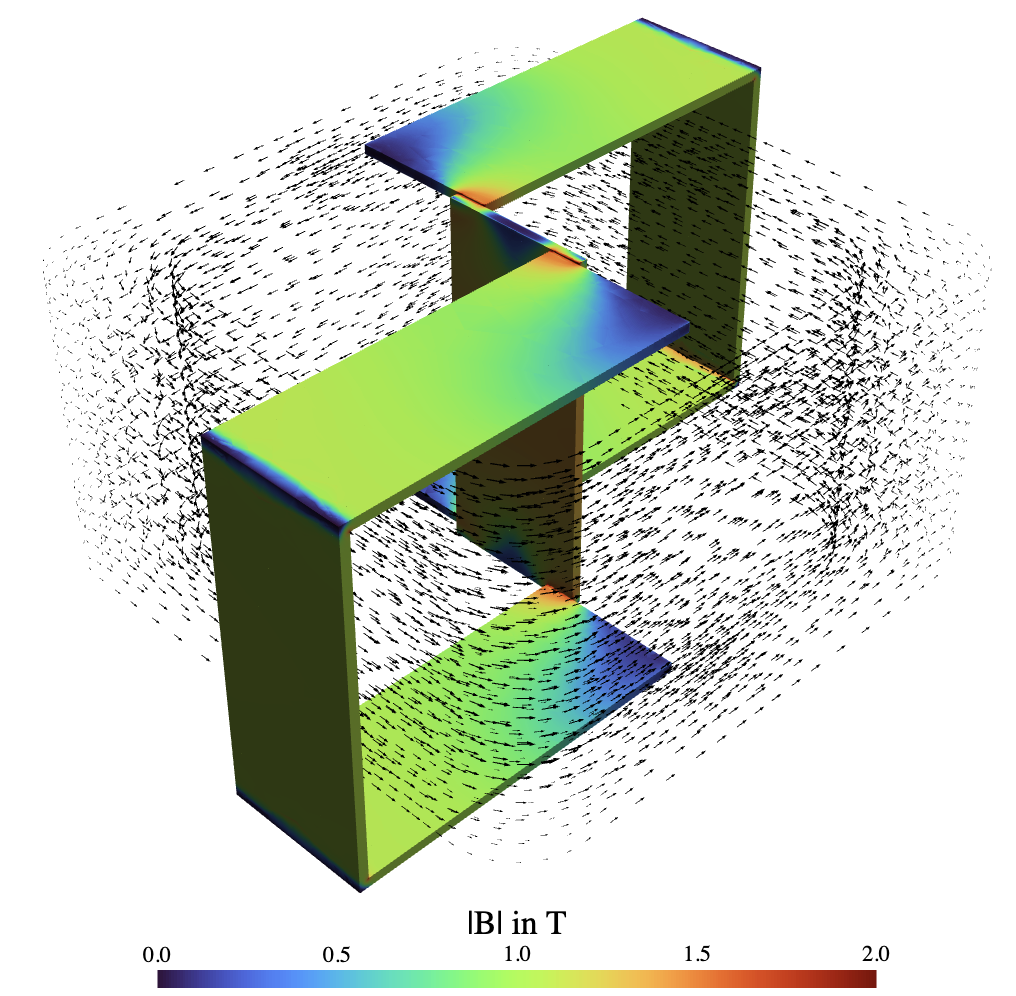}
    \caption{Solution to the magnetostatic problem using our mixed approach. Black arrows depict the direction of the current in the coil. Scalar value is the magnitude of the $\mathbf{b}$-field. (computed using \texttt{Netgen/NGSolve} and displayed using \texttt{PyVista}. }
    \label{radu:fig:2}
\end{figure}

\subsubsection*{Implementation.} 
We conduct a series of numerical tests using the vector and scalar potential formulations and the method proposed in this paper. 
Results for different mesh sizes $h$ and polynomial degree $p$ are presented. 
Since singularities occur at material interfaces, we refine the mesh towards the respective faces and edges using local mesh sizes $h_f=2^{-3} \cdot h$ and $h_e = 2^{-4} \cdot h$. The global mesh size is chosen as $h = 2^{-k}$, $k=2,3,\ldots$. The corresponding meshes are generated independently from each other. 
All tests were implemented in  \texttt{Netgen/Ngsolve}~\cite{ngsolve}. 
The same nonlinear solution strategy, i.e. a Newton method with Armijo line search, was used for all methods with the same parameters and stopping tolerances. 
The linear systems arising in every Newton step were solved with the conjugate gradient method with Jacobi preconditioner. Again, the same tolerances were used for stopping these iterations.

\subsubsection*{Numerical results.}
In Table~\ref{radu:tab:iter}, we report on the iteration numbers of the nonlinear solver required for the different methods and for different choices of the excitation current, the mesh sizes $h$, and the polynomial degrees $p$.
\begin{table}[ht!]
\centering
\setlength\tabcolsep{1.5ex}
\renewcommand{\arraystretch}{1.1}

\begin{tabular}{|c||c|c|c|c||}
\multicolumn{5}{c}{1000 ampere-turns} \\
\hline
$h$ &$p=1$ & $p=2$ & $p=3$ & $p=4$  \\
\hline
\hline
$2^{-2}$ & $8\,|\,9\,|\,8$ & $6\,|\,14\,|\,8$ & $7\,|\,13\,|\,8$ & $7\,|\,15\,|\,8$ \\ 
$2^{-3}$ & $8\,|\,12\,|\,8$ & $6\,|\,16\,|\,7$ & $7\,|\,15\,|\,8$ & $7\,|\,14\,|\,8$ \\ 
$2^{-4}$ & $9\,|\,15\,|\,7$ & $7\,|\,19\,|\,8$ & $7\,|\,16\,|\,8$ & $7\,|\,15\,|\,8$ \\ 
$2^{-5}$ & $7\,|\,17\,|\,7$ & $7\,|\,16\,|\,8$ & $7\,|\,20\,|\,8$ & $7\,|\,19\,|\,8$ \\ 
$2^{-6}$ & $7\,|\,15\,|\,7$ & $7\,|\,23\,|\,7$ & $7\,|\,22\,|\,8$ & $8\,|\,18\,|\,9$ \\
\hline 
\end{tabular}

\smallskip 

\begin{tabular}{|c||c|c|c|c||}
\multicolumn{5}{c}{3000 ampere-turns} \\
\hline
$h$ &$p=1$ & $p=2$ & $p=3$ & $p=4$  \\
\hline
\hline
$2^{-2}$ &  $8\,|\,13\,|\,8$  & $9\,|\,18\,|\,7$ & $8\,|\,16\,|\,9$ & $8\,|\,16\,|\,9$ \\ 
$2^{-3}$ & $8\,|\,14\,|\,8$ & $9\,|\,13\,|\,8$ & $8\,|\,18\,|\,8$ & $8\,|\,16\,|\,8$  \\ 
$2^{-4}$ & $8\,|\,12\,|\,8$ & $8\,|\,14\,|\,7$ & $8\,|\,15\,|\,8$ & $8\,|\,20\,|\,8$ \\ 
$2^{-5}$ & $9\,|\,17\,|\,7$ & $8\,|\,21\,|\,7$ & $8\,|\,27\,|\,8$ & $8\,|\,21\,|\,8$ \\ 
$2^{-6}$ & $8\,|\,21\,|\,7$ & $8\,|\,21\,|\,8$ & $8\,|\,20\,|\,8$ & $8\,|\,20\,|\,8$ \\
\hline
\end{tabular}
\vspace{1em}
\caption{Number of iterations for the vector potential formulation, the scalar potential approach, and the proposed modified scalar potential approach  (left to right). \label{radu:tab:iter} }
\end{table}
We observe that in general, the iteration numbers are more or less independent of the discretization parameters; see \cite{radu:Egger2024} for explanation. %
The two energy-based methods (left, right) require only about $7$--$8$ iterations to reach the desired tolerance, while the scalar potential formulation (middle), which is based on the coenergy functional, requires approximately twice as many. This is in perfect agreement with the observations made in \cite{radu:Dular2020}.
Let us note that, for the same mesh and polynomial order, the vector potential formulation has about $3$--$6$ times as many degrees of freedom as the corresponding scalar potential formulation, and in our computations, the solution of the linear systems, therefore, takes $3$--$6$ times longer.
For the proposed modified scalar potential approach, the linear solves can be done at the same complexity as for the standard scalar potential formulations. The number of Newton iterations, and hence also the total computational complexity, is reduced by about one half. 

\vspace*{-0.0em}

\subsection*{Acknowledgement}
The authors are grateful for the financial support from the international FWF/DFG-funded Collaborative Research Centre CREATOR (TRR361/SFB-F90).


\begin{thebibliography}{10}

\bibitem{radu:Albanese1997}
{\sc R.~Albanese and G.~Rubinacci}, {\em Finite element methods for the
  solution of {3D} eddy current problems}, Elsevier, 1997, pp.~1--86.

\bibitem{radu:Boffi}
{\sc D.~Boffi, F.~Brezzi, and M.~Fortin}, {\em Mixed finite element methods and
  applications}, Springer, Heidelberg, 2013.

\bibitem{radu:Ciarlet}
{\sc P.~G. Ciarlet}, {\em The finite element method for elliptic problems},
  vol.~40 of Classics in Applied Mathematics, SIAM, Philadelphia, PA, 2002.

\bibitem{radu:TEAM13}
{\sc Compumag}, {\em Team problem 13,
  \texttt{\url{www.compumag.org/wp/team/}}}.

\bibitem{radu:Dular2020}
{\sc J.~Dular, C.~Geuzaine, and B.~Vanderheyden}, {\em Finite-element
  formulations for systems with high-temperature superconductors}, IEEE Trans.
  Appl. Supercond., 30 (2020), pp.~1--13.

\bibitem{radu:Dular1997}
{\sc P.~Dular, J.-F. Remacle, F.~Henrotte, A.~Genon, and W.~Legros}, {\em
  Magnetostatic and magnetodynamic mixed formulations compared with
  conventional formulations}, IEEE Trans. Magn., 33 (1997), pp.~1302--1305.

\bibitem{radu:Egger2024}
{\sc H.~Egger, F.~Engertsberger, and B.~Radu}, {\em Global convergence of
  iterative solvers for problems of nonlinear magnetostatics},
  arXiv:2403.18520,  (2024).

\bibitem{radu:Engertsberger2023}
{\sc F.~Engertsberger}, {\em The Scalar Potential Approach in Nonlinear
  Magnetostatics}.
\newblock MSc Thesis, JKU Linz, 2023.

\bibitem{radu:Heise1994}
{\sc B.~Heise}, {\em Analysis of a fully discrete finite element method for a
  nonlinear magnetic field problem}, SIAM J. Numer. Anal., 31 (1994),
  pp.~745--759.

\bibitem{radu:Hiptmair2002}
{\sc R.~Hiptmair}, {\em Finite elements in computational electromagnetism},
  Acta Numerica, 11 (2002), pp.~237--339.

\bibitem{radu:Kelley}
{\sc C.~T. Kelley}, {\em Solving nonlinear equations with {N}ewton's method},
  Fundamentals of Algorithms, SIAM, 2003.

\bibitem{radu:Meunier}
{\sc G.~Meunier}, {\em The Finite Element Method for Electromagnetic Modeling},
  ISTE~Wiley, 2008.

\bibitem{radu:Monk}
{\sc P.~Monk}, {\em Finite element methods for {M}axwell's equations}, Oxford
  University Press, New York, 2003.

\bibitem{ngsolve}
{\sc Netgen/NgSolve}.
\newblock https://ngsolve.org/.

\bibitem{radu:Nedelec1980}
{\sc J.~C. Nédélec}, {\em Mixed finite elements in $\mathbb{R}^3$}, Numer.
  Math., 35 (1980), pp.~315--341.

\bibitem{radu:Silvester1991}
{\sc P.~P. Silvester and R.~P. Gupta}, {\em Effective computational models for
  anisotropic soft {B-H} curves}, IEEE Trans. Magn., 27 (1991), pp.~3804--3807.

\end{thebibliography}
\end{document}